\documentclass[12pt, notitlepage]{amsart}
\usepackage{latexsym, amsfonts, amsmath, amssymb, amsthm, cite}

\newtheorem{lemma}{Lemma}
\newtheorem{proposition}{Proposition}
\newtheorem{theorem}{Theorem}

\newtheorem{conjecture}{Conjecture}

\newcommand{\assign}{:=}

\newcommand{\mathd}{\mathrm{d}}
\newcommand{\mathi}{\mathrm{i}}

\usepackage{graphicx}
\usepackage{mathrsfs}

\begin{document}

\title{Large deviations in Selberg's central limit theorem}

\begin{abstract}
Following Selberg \cite{Selberg} it is known that as $T \rightarrow \infty$,
$$
\frac{1}{T} \underset{t \in [T;2T]}{\text{meas}}
\left \{ \log | \zeta(\tfrac 12 + \mathi t) |
\geq \Delta \sqrt{\tfrac 12 \log \log T} \right \}
\sim \int_{\Delta}^{\infty} e^{-u^2/2} \frac{\mathd u}{\sqrt{2 \pi}}
$$
uniformly in $\Delta \leq (\log\log\log T)^{1/2 - \varepsilon}$. 
We extend the range of $\Delta$ to
$\Delta \ll (\log \log T)^{1/10 - \varepsilon}$. We also speculate
on the size of the largest $\Delta$ for which the above normal
approximation can hold and on the correct approximation beyond
this point. 

\end{abstract}

\author{Maksym Radziwi\l\l}
\address{Department of Mathematics \\ Stanford University\\
450 Serra Mall, Bldg. 380\\
Stanford, CA 94305-2125}
\email{maksym@stanford.edu}
\thanks{The author is partially supported by a NSERC PGS-D award}

\maketitle

\section{Introduction.}


The value-distribution of $\log \zeta(\sigma + \mathi t)$  
is a classical question in the theory of the Riemann zeta function. When $\sigma > 1/2$
this distribution is well-understood and is that of an almost surely convergent
sequence of random variables (see \cite{Youness}, \cite{Matsumoto} , 
\cite{LaurinchikasBook} or \cite{Steuding}).

The half-line is special, since for $\sigma = \tfrac 12$ we 
have a \textit{central
limit theorem},
\begin{equation}
  \label{selberg} \frac{1}{T} \cdot \underset{ t \in [T;2T]} 
  {\text{meas}} \left \{
   \frac{\log | \zeta (\tfrac 12 + \mathi t)|}{\sqrt{\tfrac 12 \log\log T}} 
\geqslant \Delta \right \} =
  \int_{\Delta}^{\infty} e^{- u^2 / 2} \cdot \frac{\mathd u}{\sqrt{2
  \pi}} + o (1) \text{ , } T \rightarrow \infty
\end{equation}
originally due to Selberg \cite{Selberg}. 
Whereas the distribution of large values of $\log |\zeta(\sigma + \mathi t)|$ with
$\sigma > 1/2$ has been consistently studied since the pioneering work of Bohr and
Jessen \cite{Bohr}, the corresponding question on the half-line has only attracted more attention recently.

Conditionally on the Riemann Hypothesis, Soundararajan \cite{SoundM} obtained Gaussian upper bounds for the
left-hand side of (\ref{selberg}) (focusing mostly on the range $\Delta \gg \sqrt{\log\log T}$). 
As an application he derived near optimal upper 
bounds for moments of the Riemann zeta-function.

In this paper we will be interested in asymptotic formulas for the left-hand
side of (\ref{selberg}) when $\Delta \rightarrow \infty$ as $T \rightarrow \infty$. 
Previously asymptotic formulae of the form
\begin{equation}
  \frac{1}{T} \cdot \underset{t \in [T;2T]} {\text{meas}} 
  \left\{
  \frac{\log | \zeta ( \tfrac 12 + \mathi t) |}{\sqrt{
      \tfrac 12 \log\log T}} \geqslant \Delta \right\} \sim
  \int_{\Delta}^{\infty} e^{- u^2 / 2} \cdot \frac{\mathd u}{\sqrt{2
  \pi}} \label{asympt3}
\end{equation}
where known only in the range $\Delta \ll (\log\log\log T)^{1/2 - \varepsilon}$
as a consequence of Selberg's \cite{Tsang} near-optimal refinement of the error term in (\ref{selberg}) \footnote{Conjecturally (\ref{selberg}) admits an 
asymptotic expansion in powers of $(\log\log T)^{-1/2}$. Selberg's
obtained an
error term of
$O((\log\log\log T)^2 (\log\log T)^{-1/2})$ in (\ref{selberg}),
which is thus close to optimal in this sense.}
In this paper we introduce a new method that allows us to extend
(\ref{asympt3}) to the large-deviations range $\Delta \ll (\log\log T)^{\alpha}$ for
some small but fixed $\alpha > 0$.
\begin{theorem}
  \label{Theorem1}The asymptotic formulae (\ref{asympt3}) holds for $\Delta \ll
  (\log\log T)^{1 / 10 - \varepsilon}$.
\end{theorem}
Our method is very versatile and allows to extend most of the known distribution 
results about $\zeta(\sigma + \mathi t)$ (with $\sigma = 1/2 + o_{T \rightarrow \infty}(1)$) to a 
large deviations setting (for example \cite{Bourgade} or the joint value distribution
of $\Re \log \zeta(\tfrac 12 + \mathi t)$ and $\Im \log \zeta(\tfrac 12 + \mathi t)$).
The method also adapts to the
study of the distribution of large values of additive functions over sets of integers 
where only moments (but not moment generating functions) are available. 

To prove Theorem \ref{Theorem1} we use Selberg's work to reduce the problem to a question
about Dirichlet polynomials. Theorem \ref{Theorem1} then follows from the Proposition below,
which might be of independent interest.

\begin{proposition}
  \label{Proposition2}Let $x = T^{1 /
  (\log\log T)^2}$. Then, uniformly in $\Delta = o (
  \sqrt {\log\log T})$,
  \begin{equation}
    \frac{1}{T} \cdot \underset{t \in [T;2T]} {\text{meas}} 
    \left\{ \Re
    \sum_{p \leqslant x} \frac{1}{p^{\tfrac 12 + \mathi t}} \geqslant \Delta
    \cdot \sqrt{\tfrac 12 \sum_{p \leqslant x} \frac{1}{p}} \right\} \sim
    \int_{\Delta}^{\infty} e^{- u^2 / 2} \cdot \frac{\mathd u}{\sqrt{2
    \pi}} \label{asympt2}
  \end{equation}
\end{proposition}

Previously a result such as Proposition \ref{Proposition2} was known only for
fairly short Dirichlet polynomials (i.e $x \leqslant (\log T)^{\theta}$) (see
\cite{LaurinchikasI} and also \cite{LaurinchikasII} for related work). The extension
in the length of the Dirichlet polynomial is responsible for our improvement 
in Theorem \ref{Theorem1}.

The key idea in 
our proof of Proposition \ref{Proposition2}
is to work on a subset $A$ of $[T;2T]$
on which the Dirichlet polynomial does not attain large values. We explain the idea 
briefly in more details in subsection (\ref{explanation}) below. 

Proposition \ref{Proposition2} can be extended to the range $\Delta
\asymp \sqrt {\log\log T}$.  In that situation we obtain the following result.
\begin{proposition}
\label{Proposition3}
Let $x = T^{1/ (\log\log T)^2}$ and $k > 0$. Uniformly
in $\Delta \sim k \sqrt{\tfrac 12 \log\log T}$,
   \begin{equation*}
    \frac{1}{T} \cdot \underset{t \in [T;2T]} {\text{meas}} 
    \left\{ \Re
    \sum_{p \leqslant x} \frac{1}{p^{\tfrac 12 + \mathi t}} \geqslant \Delta
    \cdot \sqrt{\tfrac 12 \sum_{p \leqslant x} \frac{1}{p}} \right\} \sim
    c_k \int_{\Delta}^{\infty} e^{- u^2 / 2} \cdot \frac{\mathd u}{\sqrt{2
    \pi}} 
  \end{equation*}
with $0 < c_k \neq 1$ a constant, depending only on $k$.  
\end{proposition}

It is thus
apparent that (\ref{asympt2}) does not persist for $\Delta \geqslant
\varepsilon \sqrt{\log\log T}$. Similarly,
\begin{equation}
  \frac{1}{T} \cdot \underset{t \in [T;2T]} {\text{meas}} 
  \left\{
  \frac{\log | \zeta ( \tfrac 12 + \mathi t) |}{\sqrt{
      \tfrac 12 \log\log T}} \geqslant \Delta \right\} \sim
  \int_{\Delta}^{\infty} e^{- u^2 / 2} \cdot \frac{\mathd u}{\sqrt{2
  \pi}} \label{asympt}
\end{equation}
cannot be true for $\Delta \geqslant \varepsilon
\sqrt{\log\log T}$ with $\varepsilon > 0$ small.
Indeed, one can show (on the Riemann Hypothesis) that this would 
contradict the moment conjectures,
\begin{equation}
  \int_T^{2 T} | \zeta (\tfrac 12 + \mathrm{i} t) |^{2 k} \mathrm{d} t \sim C_k
  \cdot T (\log T)^{k^2} \label{moment_conjecture2}
\end{equation}
$(T \rightarrow \infty)$ since it is conjectured that $C_k \neq 1$ for $k \in (0, 1)$ 
and if (\ref{asympt}) holds for $\Delta \leq \varepsilon \sqrt{\log\log T}$ then
$C_k = 1$ for $k \leq \varepsilon$ (the Riemann Hypothesis is used because we
appeal to \cite{SoundM} to first restrict the range of integration in (\ref{moment_conjecture2}) to those $t$'s at which $\log |\zeta(\tfrac 12 + \mathi t)|
\sim k \log\log T$).  
Nonetheless, by analogy to Proposition \ref{Proposition2} we expect
(\ref{asympt}) to hold for all smaller $\Delta$.

\begin{conjecture}
  \label{Conjecture1}If $\Delta = o (
  \sqrt{\log\log T})$ then (\ref{asympt}) holds. 
\end{conjecture}

For larger values of $\Delta$ -- in analogy to Proposition
\ref{Proposition3} -- we conjecture that (\ref{asympt}) deviates from the truth
only by a ``constant multiple''. 

\begin{conjecture}
  \label{Conjecture2}Let $k > 0$. If $\Delta \sim k
  \sqrt {\log\log T}$, then
  \[ \frac{1}{T} \cdot \underset{t \in [T;2T]}{\text{meas}} 
  \left\{ \frac{\log | \zeta (\tfrac 12 + \mathi t) |}{\sqrt{
      \tfrac 12 \log\log T}} \geqslant \Delta \right\} \sim
     C_k \cdot \int_{\Delta}^{\infty} e^{- u^2 / 2} \cdot \frac{\mathd
     u}{\sqrt{2 \pi}} \]
  with $C_k$ the same constant as in (\ref{moment_conjecture2})
\end{conjecture}
An abelian argument 
(conditional on RH and using \cite{SoundM} to truncate the tails) shows that
if the left-hand side of (\ref{asympt}) is at all asymptotic to a constant
$\kappa$ times a standard Gaussian in the range
$\Delta \sim k (\tfrac 12 \log\log T)^{1/2}$ then $\int_{T}^{2T} |\zeta(\tfrac
12 + \mathi t)|^{2k} \mathd t \sim \kappa T (\log T)^{k^2}$, hence
$\kappa = C_k$. Thus there is only one reasonable choice for the constant in 
Conjecture 2 above.
\par
For negative values of $k$ Conjecture \ref{Conjecture2} is likely to be false
as soon as $k \leqslant - 1 / 2$. Indeed in that range the zeros of $\zeta
(s)$ should force a transition to an exponential distribution, see \cite{Hughes} for a precise statement.

\subsection{Remarks}

\begin{enumerate}

\item There is no difficulty in adapting our
proof to the study of negative values of $\log | \zeta (\tfrac 12 + \mathi t)
|$. Also, our argument carries over to $S (t) := \frac{1}{\pi} \Im
\log \zeta(\tfrac 12 + \mathi t)$ with little to no changes.

\item Conditionally on the Riemann Hypothesis we
obtain the better range $\Delta \leqslant
(\log\log T)^{1 / 6 - \varepsilon}$ for $S(t)$.
(assuming the Riemann Hypothesis, the analogue of Lemma \ref{Selberg_lemma}
for $S(t)$ has an $k^{2k}$ instead of $k^{4k}$. This is responsible for the
improvement).

\item Assuming the relevant Riemann Hypothesis our result extends to
elements of the Selberg class (see \cite{Selbergclass} for a
definition).

\item Finally, our method covers the case of the joint distribution
of $\log |\zeta( \tfrac 12 + \mathi t)|$ and $\Im \log \zeta(\tfrac 12 + \mathi t)$. 
Similarly, one can recover Bourgade's \cite{Bourgade} result on the joint distribution of shifts
of the Riemann zeta function in a large deviation setting. 

\item The term $k^{4k}$ in Lemma \ref{Selberg_lemma}  is directly responsible the range $\Delta \ll (\log\log T)^{1/10 - \varepsilon}$ in Theorem \ref{Theorem1}.
For example, an  improvement of the $k^{4k}$ to $k^{2k}$ in 
Lemma \ref{Selberg_lemma} 
would give Theorem \ref{Theorem1} uniformly in $\Delta \ll (\log\log T)^{1/6 
- \varepsilon}$. 

\end{enumerate}

\subsection{Outline of the proof of Proposition 1}
\label{explanation}

The key idea in the proof of Proposition
\ref{Proposition2} is to initially work on a subset $A$ of $[T ; 2 T]$ on
which the Dirichlet polynomial does not attain large values (say $\leq c \cdot
\log\log T$). This allows us to express the moment
generating function,
\begin{equation}
  \int_A \exp \left( z \cdot \Re \sum_{p \leqslant x} \frac{1}{p^{1 /
  2 + \mathi t}} \right) \cdot \mathd t \text{ \ , \ } |z| \ll 1
  \label{momentgenerating}
\end{equation}
in terms of only the first $\asymp \log\log T$
moments of $\Re \sum_{p \leqslant x} p^{- 1 / 2 - \mathi t}$ over
$t \in A$. These moments can be taken over the full interval $[T ; 2 T]$
(since $A$ is very close in measure to $T$) and then they become easy to
estimate. On adding up the contribution from the moments we obtain a very precise
estimate for (\ref{momentgenerating}). From there,
by standard probabilistic techniques, we obtain (\ref{asympt2}) with $t$ restricted
to $A$. Since $A$ is very close in measure to $T$, we get $(3)$
without the restriction to $t \in A$. 

\textbf{Acknowledgments.} The author would like to thank his supervisor K. Soundararajan
for advice and encouragements.

\textbf{Notation.} Throughout $\varepsilon$ will denote an arbitrary small
but fixed positive real. We allow $\varepsilon$ to differ from line to line.
Finally $\log_{k} T$ denotes the $k$-th iterated natural logarithm, so that
$\log_{k} \assign \log \log_{k-1} T$ and $\log_{1} T = \log T$. 

\section{Lemmata}

\begin{lemma}
  [Selberg, \cite{Tsang}] \label{Selberg_lemma} Uniformly in $k \geqslant 0$,
  \[ \int_T^{2 T} \left | \log | \zeta (\tfrac 12 + \mathrm{i} t) | -\mathfrak{R}
     \sum_{p \leqslant x} \frac{1}{p^{\tfrac 12 + \mathrm{i} t}} \right|^{2 k}
     \mathrm{d} t \ll A^k \cdot k^{4 k} +
     A^k \cdot k^k \cdot
     (\log\log\log T)^k \]
   where $x = T^{1 / (\log\log T)^{2}}$ and $A>0$ constant. 
\end{lemma}

\begin{proof}
See Tsang's thesis \cite{Tsang}, page 60. 
\end{proof}

\begin{lemma}
  [Hwang, \cite{Hwang}] \label{Hwang} Let $W_n$ be a sequence of distribution functions such
that for $|s| \leqslant \varepsilon$, 
  \[ \int_{- \infty}^{\infty} e^{st} \mathrm{d} W_n (t) = F(s) \exp \left(
     \frac{s^2}{2} \cdot \phi (n) \right) \cdot \left( 1 + O \left(
     \kappa_n^{- 1} \right) \right) \]
     with $F(s)$ a function analytic around $s=0$ and equal to $1$ at $s=0$. 
Then, uniformly in $\Delta \leqslant o (\min (\kappa_n, \sqrt{\phi (n)})$,
  \[ 1 - W_n (\Delta \sqrt{\phi(n)}) \sim \int_{\Delta}^{\infty} e^{- u^2 / 2} \cdot
     \frac{\mathrm{d} u}{\sqrt{2 \pi}} \]
\end{lemma}

\begin{proof}This is a special case of Theorem 1 in \cite{Hwang}.
\end{proof}

\begin{lemma} [Soundararajan, \cite{SoundM}]
  \label{SoundLemma} For $x \leqslant T^{1 / 4k}$,
  \[ \int_T^{2 T} \left| \mathfrak{R} \sum_{p \leqslant x} \frac{1}{p^{\tfrac 12 +
     \mathrm{i} t}} \right|^{2 k} \mathrm{d} t \ll k! \cdot T \left( \sum_{p
     \leqslant x} \frac{1}{p} \right)^k \]
\end{lemma}

\begin{proof} This is a special case of Lemma 3 in \cite{SoundM} \end{proof}. 

\section{ Proof of Proposition 1 }

\begin{lemma} \label{Lemma1}
Let $p_1 , \ldots, p_k \leq x$ be primes. Then,
$$
\frac{1}{T} \int_{T}^{2T} \prod_{\ell = 1}^{k} \cos (t \log p_\ell) \mathd t
= f(p_1 \cdot \ldots 
\cdot p_k) + O(2^k x^k)
$$
where $f$ is a multiplicative function defined by
$
f(p^{\alpha}) = \frac{1}{2^{\alpha}} \binom{\alpha} {\alpha / 2}
$
By convention the binomial coefficient is zero when $\alpha/2$ is
not an integer. 
\end{lemma}

\begin{proof}

Write $n = p_1 \cdot \ldots \cdot p_k = q_1^{\alpha_1} \cdot \ldots q_r^{\alpha_r}$
with $q_i$ mutually distinct primes. Notice that
\begin{eqnarray*}
\cos(t \log q_i)^{\alpha_i} & = & \frac{1}{2^{\alpha_i}} \cdot 
\left ( e^{\mathi t \log q_i} + e^{- \mathi t \log q_i} \right )^{\alpha_i} \\ 
 & = &
\frac{1}{2^{\alpha_i}} \binom{\alpha_i}{\alpha_i / 2} + 
  \sum_{\alpha_i /2 \neq \ell \leq \alpha_i}
\frac{1}{2^{\alpha_i}} \binom{\alpha_i}{\ell} e^{\mathi (\alpha_i - 2\ell) t \log q_i} 
\end{eqnarray*}
Note that the leading coefficients in front of every $e^{\mathi(\alpha_i - 2\ell)}$ is $\leq 1$ in
absolute value. Furthermore $|\alpha_i - 2\ell| \leq \alpha_i$. Therefore
$$
\frac{1}{T} \int_{T}^{2T} \prod_{\ell = 1}^{k} \cos(t \log p_\ell) \mathd t
= \frac{1}{T} \int_{T}^{2T} \prod_{i = 1}^{r} \cos(t \log q_{i})^{\alpha_i} \mathd t
= f(n )
+ (\ldots)
$$
where inside $(\ldots)$ there is a sum of $\leq (\alpha_1 + 1) \cdot \ldots
\cdot (\alpha_r + 1) \leq 2^{\alpha_1} \cdot \ldots \cdot 2^{\alpha_r} = 2^k$ 
terms of the form $\gamma e^{\mathi t 
(\beta_1 \log q_1 + \ldots \beta_r \log q_r)}$
with $|\gamma| \leq 1$ and integers $|\beta_i| \leq \alpha_i$. 
Since $\beta_1 \log q_1 + \ldots + \beta_r \log q_r \gg x^{-k}$
the integral over $T \leq t \leq 2T$ of each of these terms is at 
most $\ll x^k$ . Since there is $\leq 2^k$ of these terms and each has 
a leading coefficient $|\gamma| \leq 1$ the integral over 
$T \leq t \leq 2T$ of all the terms inside $(\ldots)$ 
contributes at most $O(2^k x^k)$.
\end{proof}

\begin{lemma}
  \label{Lemma2}Let $x = T^{1 / (\log\log T)^2}$.
  Then,
  \[ \frac{1}{T} 
  \int_T^{2 T} \left( \mathfrak{R} \sum_{p \leqslant x} \frac{1}{p^{\tfrac 12 +
     \mathrm{i} t}} \right)^k \mathrm{d} t = \frac{1}{2 \pi \mathrm{i}} \oint
     \prod_{p \leqslant x} I_0 \left( \frac{z}{\sqrt{p}} \right) \cdot
     \frac{\mathrm{d} z}{z^{k + 1}} + O \left( (2 x)^{k} \right) \]
  where $I_0 (z) = (1 / \pi) \int_0^{\pi} e^{z \cos \theta} \mathrm{d} \theta
  = \sum_{n \geqslant 0} (z / 2)^{2 n} \cdot 1 / (n!)^2$ is the modified 0-th
  order Bessel function. 
\end{lemma}
\begin{proof}
Given an integer $n = p_1^{\alpha_1} \cdot \ldots \cdot
p_{\ell}^{\alpha_{\ell}}$ with $\Omega (n) = k$ and $p_i \leqslant x$ there
are $k! / (\alpha_1 ! \ldots \alpha_{\ell} !)$ ways in which this integer can
be written as a product of $k$ primes. We define a multiplicative function $g
(n)$ by $g (p^{\alpha}) = 1 / \alpha !$ and $g (p^{\alpha}) = 0$, $p > x$, so
that $k!g (n)$ is equal to the number of ways in which an integer $n$ with
$\Omega (n) = k$ can be expressed as a product of $k$ primes $\leqslant x$.

By Lemma \ref{Selberg_lemma} and the above observation,
\begin{eqnarray*}
  \frac{1}{T} \int_T^{2 T} \left( \mathfrak{R} \sum_{p \leqslant x} \frac{1}{p^{\tfrac 12 +
  \mathi t}} \right)^k \mathd t & = & \sum_{p_1, \ldots, p_k \leqslant x}
  \frac{f (p_1 \cdot \ldots \cdot p_k)}{\sqrt{p_1 \cdot \ldots \cdot p_k}} + O
  \left( (2 x)^{2 k} \right)\\
  & = & k! \sum_{\Omega (n) = k} \frac{f (n)}{\sqrt{n} } \cdot g (n) + O
  \left( (2 x)^{2 k} \right)
\end{eqnarray*}
We detect the condition $\Omega (n) = k$ by using Cauchy's integral
formula. Thus the above is equal to
\[ \frac{k!}{2 \pi \mathi} \oint \sum_{n \geqslant 1} \frac{f (n)}{\sqrt{n}}
   \cdot g (n) z^{\Omega (n)} \cdot \frac{\mathd z}{z^{k + 1}} + O \left( (2
   x)^{2 k} \right) \]
Since the functions $f, g$ and $z^{\Omega (n)}$ are multiplicative, the above
sum over $n \geqslant 1$ does factor into an Euler product,
\[ \prod_{p \leqslant x} \left( 1 + \sum_{\ell \geqslant 1} 
\left ( \frac{z}{2\sqrt{p}} \right)^{2\ell}
\frac{1}{(2 \ell) !} \cdot 
   \binom{2\ell}{\ell} \right) = \prod_{p \leqslant x} I_0 \left( \frac{z}
         {\sqrt{p}}
   \right) \]
as desired. \end{proof}

We are now ready to prove Proposition \ref{Proposition2}.

\begin{proof} [Proof of Proposition \ref{Proposition2}] By Lemma \ref{SoundLemma},
\[ \int_T^{2 T} \left| \mathfrak{R} \sum_{p \leqslant x} \frac{1}{p^{\tfrac 12 +
   \mathi t}} \right|^{2 k} \mathd t \ll Tk! \cdot \left( \sum_{p \leqslant x}
   \frac{1}{p} \right)^k \]
Therefore the set of those $t \in [T ; 2 T]$ for which $|\mathfrak{R} \sum_{p
\leqslant x} p^{- \tfrac 12 - \mathi t} | \geqslant \log_2 T$ has measure
$\ll T (k / \log_2 T)^k$ . Choosing $k = \lfloor \log_2 T / e \rfloor$, this
measure is $\ll T (\log T)^{- \delta}$ where $\delta = 1 / e$, which is
negligible (also, the exact value of $\delta$ is unimportant). 
Thus, we can restrict our attention to the set 
$$
A \assign \left \{ t \in [T;2T] :  \left | \Re \sum_{p \leq x} \frac{1}{p^{\tfrac 12 +
\mathi t}} \right | \leq \log_{2} T \right \}
$$
We denote its complement in $[T;2T]$ by $A^{c}$. 
For complex $|z| \leqslant \varepsilon < \tfrac {1}{100}$,
\begin{equation}
  \int_A \exp \left( z\mathfrak{R} \sum_{p \leqslant x} \frac{1}{p^{\tfrac 12 +
  \mathi t}} \right) \mathd t 
  = \sum_{k \leqslant 3\mathcal{V}} \frac{z^k}{k!} \int_A
  \left( \mathfrak{R} \sum_{p \leqslant x} \frac{1}{p^{\tfrac 12 + \mathi t}}
  \right)^k \mathd t + O (T (\log T)^{- 3}) \label{gf}
\end{equation}
where $\mathcal{V} = \log\log T$. The error term arises from bounding the
terms with $k \geqslant 3\mathcal{V}$: each contributes at most $ T (\varepsilon
\log\log T)^k / k! \leq T e^{-k}$ since the integral over the set $A$ is less than
$T (\log\log T)^k$ by definition of $A$. To compute the moments
with $k \leqslant 3\mathcal{V}$ we use Cauchy's inequality and notice that
\begin{multline}
  \left | \int_A \left( \mathfrak{R} \sum_{p \leqslant x} \frac{1}{p^{\tfrac 12 +
  \mathi t}} \right)^k \mathd t - \int_T^{2 T} \left( \mathfrak{R} \sum_{p
  \leqslant x} \frac{1}{p^{\tfrac 12 + \mathi t}} \right)^k \mathd t \right |
  \leqslant 
  \\ \leqslant  \sqrt{\text{meas} \left( A^c \right)} \cdot \left(
  \int_T^{2 T} \left| \mathfrak{R} \sum_{p \leqslant x} \frac{1}{p^{\tfrac 12 +
  \mathi t}} \right|^{2 k}_{} \mathd t \right)^{\tfrac 12}  
  \ll T \cdot (\log T)^{- \delta / 2} \cdot \sqrt{k!} \cdot
  (\log\log T)^{k / 2}
\end{multline}
by Lemma \ref{SoundLemma}. The integral over $T \leqslant t \leqslant 2 T$
is readily available through Lemma \ref{Lemma2}. Thus,
\[ \int_A \left( \mathfrak{R} \sum_{p \leqslant x} \frac{1}{p^{\tfrac 12 + \mathi
   t}} \right)^k \mathd t = \frac{Tk!}{2 \pi \mathi} \oint \prod_{p \leqslant
   x} I_0 \left( \frac{w}{\sqrt{p}} \right) \cdot \frac{\mathd w}{w^{k + 1}} + O
   \left( \frac{T \sqrt{k!} \cdot (\log\log T)^{k / 2}}{(\log T)^{\delta / 2}} \right) \]
We choose the contour to be a circle of radius $r = 2$ (say) around the
origin. Summing the above over $k \leqslant 3\mathcal{V}$ we conclude that
(\ref{gf}) is equal to
\begin{eqnarray*}
  &  & \frac{T}{2 \pi \mathi} \oint \prod_{p \leqslant x} I_0 \left( \frac{
  w}{\sqrt{p}} \right) \sum_{k \leqslant 3\mathcal{V}} \frac{z^k}{w^{k + 1}} \cdot
  \mathd w + O \left( T (\log T)^{- \delta/2 + \varepsilon} \right)
\end{eqnarray*}
Since $\prod_{p \leqslant x} I_0 (w / \sqrt{p}) \ll (\log T)^2$ on $|w| = 2$, and
$\sum_{k \geqslant 3\mathcal{V}} |z / w|^k \ll (\log T)^{- 3 \log 2 + \varepsilon}$ we can
complete the sum over $k \leqslant 3\mathcal{V}$ to all $k$, making an error
of $T (\log T)^{- \delta / 5}$. Thus the above (and hence (\ref{gf})) is equal to
\begin{equation}
\label{Bessel}
  \frac{T}{2 \pi \mathi} \oint \prod_{p \leqslant x} I_0 \left( \frac{
  w}{\sqrt{p}} \right) \cdot \frac{\mathd w}{w - z} + O \left( T (\log T)^{- \delta /
  5} \right)
  = T \cdot \prod_{p \leqslant x} I_0 \left( \frac{z}{\sqrt{p}} \right) + O
  \left( T (\log T)^{- \delta / 5} \right)
\end{equation}
by Cauchy's formula. But,
$
\prod_{p \leqslant x} I_0 \left (\frac{z}{\sqrt{p}} \right ) = F(z) \cdot \exp \left ( \frac{z^2}{2} \cdot 
\tfrac{1}{2} \sum_{p \leq x}
\frac{1}{p} \right ) \cdot ( 1 + O(1/\log T))
$
with $F(z)$ analytic around $z = 0$ and equal to $1$ at $z = 0$. 
Since $|z| \leqslant \varepsilon$
it is clear that the error term in (\ref{Bessel}) can be made relative at a small cost in the 
error term. We conclude that
$$
\int_{A} \exp \left ( z \cdot \Re \sum_{p \leq x} \frac{1}{p^{\tfrac 12 + \mathi t}}
\right ) \mathd t = T F(z) \cdot \exp \left ( \frac{z^2}{2} \cdot \tfrac 12
\sum_{p \leq x} \frac{1}{p} \right ) \cdot ( 1 + O((\log T)^{-\delta / 10} )
$$
uniformly in $|z| \leq \varepsilon$. 
Thus Lemma
\ref{Hwang} is applicable, and it follows that
\[ \frac{1}{\text{meas}(A)} \cdot \underset{t \in A}{\text{meas}} \left\{ \mathfrak{R}
   \sum_{p \leqslant x} \frac{1}{p^{\tfrac 12 + \mathi t}} \geqslant \Delta \cdot
   \sqrt{\tfrac 12 \sum_{p \leqslant x} \frac{1}{p}} \right\} \sim
   \frac{T}{meas(A)} \int_{\Delta}^{\infty} e^{- u^2 / 2} \cdot \frac{\mathd u}{\sqrt{2 \pi}}
\]
Now $\text{meas}(A^c) \ll T (\log T)^{-\delta}$ while $\int_{\Delta}^{\infty} e^{-u^2/2} \frac{\mathd t}{\sqrt{2\pi}} \gg (\log T)^{-o(1)}$ (because $\Delta = o(\sqrt{\log\log T}$). Hence the preceding equation becomes
\[ \frac{1}{T} \cdot \underset{t \in [T ; 2 T]}{\text{meas}} \left\{
   \mathfrak{R} \sum_{p \leqslant x} \frac{1}{p^{\tfrac 12 + \mathi t}} \geqslant
   \Delta \cdot \sqrt{\tfrac 12 \sum_{p \leqslant x} \frac{1}{p}} \right\} \sim
   \int_{\Delta}^{\infty} e^{- u^2 / 2} \cdot \frac{\mathd u}{\sqrt{2 \pi}}
\]
This establishes our claim\end{proof}


\section{ Proof of Theorem 1.}

\begin{proof}[Proof of Theorem \ref{Theorem1}] Let $x = T^{1 / (\log\log T)^2}$.
Let $1/4 > \delta > 0$ to be fixed later. By Lemma \ref{Selberg_lemma} the
set of those $t \in [T ; 2 T]$ for which
$$
  \left | \log | \zeta (\tfrac 12 + \mathi t) | -\mathfrak{R} \sum_{p \leqslant x}
  \frac{1}{p^{\tfrac 12 + \mathi t}} \right | \geqslant \mathcal{L} \assign
  (\log\log T)^{\tfrac 14 + \delta}
$$
has measure $\ll T \cdot \mathcal{L}^{- 2 k} \cdot A^k \cdot k^{4 k}
\cdot (\log_3 T)^k$ for all $k \geq \log_3 T$. Choosing $k = \lfloor 
\mathcal{L}^{1/2}/ A e \rfloor$
we obtain a measure of $\ll T \exp (- c \mathcal{L}^{1/2}$) for some
constant $c > 0$. Thus except for a
set of measure $\ll T \exp (-c \mathcal{L}^{1/2})$,
$$
\log |\zeta(\tfrac 12 + \mathi t)| = \Re \sum_{p \leq x} \frac{1}{p^{\tfrac 12 + \mathi t}}
+ \theta \mathcal{L}
$$
with a $|\theta| \leqslant 1$ depending on $t$. The measure of those $t \in [T;2T]$ 
for which
\begin{equation}
\label{main_ineq}
\Re \sum_{p \leq x} \frac{1}{p^{\tfrac 12 + \mathi t}} \geqslant \Delta \cdot \sqrt{\tfrac 12 
\cdot \log \log T} + \theta \mathcal{L}
\end{equation}
is (since $\theta \leq 1$) at least the measure of those $t \in [T;2T]$
for which 
$$\Re \sum_{p \leq x} \frac{1}{p^{\tfrac 12 + \mathi t}}
\geq \Delta \cdot \sqrt{\tfrac 12 \log\log T} + \mathcal{L}$$ 
and this measure is at least, 
\begin{equation} \label{lo}
T (1 + o(1)) \int_{\Delta + 2\mathcal{L}/ \sqrt{\log\log T}}^{\infty} e^{- u^2
   / 2} \cdot \frac{\mathd u}{\sqrt{2 \pi}}
\end{equation} by Proposition \ref{Proposition2}. 
Similarly (using this time $\theta \geq -1$) the measure of those
$t \in [T;2T]$ for which (\ref{main_ineq}) holds is at most
\begin{equation} \label{up}
T (1 + o(1)) \int_{\Delta - 2\mathcal{L} / \sqrt{\log\log T}}^{\infty} e^{- u^2 / 2} \cdot \frac{\mathd u}{\sqrt{2 \pi}}.
\end{equation}
When $\Delta \mathcal{L} = o((\log\log T)^{1/2})$ both (\ref{lo}) and (\ref{up})
are equal to $T \int_{\Delta}^{\infty} e^{- u^2 / 2} \cdot \frac{\mathd
u}{\sqrt{2 \pi}} \cdot (1 + o (1))$. We conclude that,
\[ \frac{1}{T} \cdot \underset{t \in [T ; 2 T]}{\text{meas}} \left\{
   \frac{\log | \zeta (\tfrac 12 + \mathi t) |}{\sqrt{\tfrac 12 \log\log T}}
   \geqslant \Delta \right \} = \frac{(1 + o (1))}{\sqrt{2 \pi}}
   \int_{\Delta}^{\infty} e^{- u^2 / 2} \mathd u + O \left( e^{-c \mathcal{L}^{1
   / 2}} \right) \]
provided that $\Delta \mathcal{L}= o (\log\log T)^{1/2}$. The main term
dominates the error term as long as 
$\Delta \leqslant \mathcal{L}^{1/4 - \varepsilon}$.
Both conditions $\Delta \leqslant \mathcal{L}^{1/4 - \varepsilon}$ and $\Delta
\mathcal{L}= o (\log\log T)^{1/2}$ are met if we choose $\delta = 3 /
20$ and require that $\Delta \leqslant (\log\log T)^{1 / 16 + \delta / 4 - \varepsilon} = (\log \log T)^{1/10 - \varepsilon}$.
\end{proof}

\section{Proof of Proposition 2}

We will only give a brief sketch.

\begin{lemma}
Let $C > 0$ be given. Then, 
uniformly in $0 \leq \Re z \leq C$, and $4C^2 \leq |\Im z|^2 \leq x^{1/8}$,
$$
\prod_{p \leq x} 
I_0 \left (\frac{z}{\sqrt{p}}\right) \ll \exp(- c (\Im z)^{2})
$$
for some constant $c > 0$. 
\end{lemma}

\begin{proof} 
For $|z|^2 \leq x$ we write
$$
\prod_{p \leq x} I_0 \left ( \frac{z}{\sqrt{p}} \right )
= \prod_{p \leq |z|^2} I_0 \left ( \frac{z}{\sqrt{p}} \right )
\cdot \prod_{|z|^2 \leq p \leq x} I_0 \left ( \frac{z}{\sqrt{p}} \right ) 
e^{-(z/2\sqrt{p})^{2}} \cdot \exp \left ( -\frac{z^2}{2} \cdot \frac{1}{2}
\sum_{|z|^2 \leq p \leq x} \frac{1}{p} \right )
$$
Using the bound $I_0(z) \ll \exp(-|\Im z|)$ the first term is 
$
\ll \exp(- c |z|^{3/2} / \log |z|)
$
Using the bound $I_0 (z) e^{-(z/2)^2} \ll \exp(|z|^{4})$ we get that the
contribution of the second term is $\ll \exp(c |z|^2 / \log |z|)$. Finally
the third term is a Gaussian and thus contributes
$
\ll \exp( (1/2) ((\Re z)^2 - (\Im z)^2) (\log (\log x / 2 \log |z|))
$
. Under our assumptions on $z$ and $x$ 
the last bound dominates and the claim follows. 
\end{proof} 


\begin{proof}[Proof of Proposition 2]
By a small modification of the proof of Proposition \ref{Proposition2}, for any fixed $\delta > 0$, uniformly in $z$, 
\begin{equation*}
\int_{A} \exp \left ( z \cdot \Re \sum_{p \leq x} \frac{1}{p^{\tfrac 12 + \mathi t}} \right ) \cdot \mathd t
= T \prod_{p \leq x} I_0 \left ( \frac{z}{\sqrt{p}} \right )  + O(
(T \log T)^{-k^2-1})
\end{equation*}
with the set $A$ and the function $F(\cdot)$ as in the proof of Proposition \ref{Proposition2}, except that now $A$ is the set
of those $t$'s at which the Dirichlet polynomial $\Re \sum_{p \leq x} p^{-1/2 - \mathi t}
\leq C \log\log T$ with some $C$ large enough
but fixed. The constant $C$ is choosen so large so as to guarantee an error
of $T (\log T)^{-k^2-1}$ above and
$\text{meas}(A^c) \ll T (\log T)^{-k^2 - 1}$. 

Let $H(z;x) = \prod_{p \leq x} I_0 (z / \sqrt{p}) \exp(-z \Delta \sigma(x))$
where $\sigma(x) = (\tfrac 12 \sum_{p \leq x} \frac{1}{p})^{1/2}$.  
By Tenenbaum's formula (see equation (10) and above in \cite{Tenenbaum}),
\begin{multline}
\label{saddle}
\underset{t \in A}{\text{meas}} \left \{ \Re \sum_{p \leq x} \frac{1}{p^{\tfrac 12 + \mathi t}} \geq
\Delta \sqrt{\tfrac 12 \sum_{p \leq x} \frac{1}{p}} \right \} = 
\frac{\text{meas}(A)}{2\pi\mathi} 
\int_{c - \mathi \infty}^{c + \mathi \infty} 
H(s;x)
\frac{T \mathd s}{s ( s + T)} + \\
+  O \left ( \left | \frac{T}{2\pi} 
\int_{c - \mathi\infty}^{c + \mathi \infty} 
\frac{H(s;x) T \mathd s}{(s + T) (s + 2T)}
\right | \right ) + O(T (\log T)^{-k^2 - 1})
\end{multline}
where $T = x^{1/16}$ and $c = \Delta \cdot ( \sigma(x) )^{-1/2}$. 
By the previous Lemma (used to handle the range $\Im s < x^{1/8}$ 
and the fast decay of $1/s (s + T)$ (used to handle $\Im s > x^{1/8}$) 
we can truncate the above integral
at $|\Im s| \asymp \psi(x) / \sigma(x)$ with $\psi(x) \rightarrow
\infty$ very slowly. This induces an negligible error
term of $o(\int_{\Delta}^{\infty} e^{-u^2/2} \mathd u)$. Furthermore, since
$T$ is so large, the $O(\cdot)$ term in the above equation is also 
negligible. Write $
\prod_{p \leq x} I_0 ( z / \sqrt{p} ) = F(z)
\cdot \exp(\frac{z^2}{2} \cdot \sigma(x)^2)
$
with $F(z)$ analytic. Then, on completing the square and
changing variables the first integral above 
(truncated at $|\Im s| \asymp \psi(x)/\sigma(x)$) can be rewritten as
$$
\text{meas}(A) \cdot \frac{e^{-\Delta^{2}/2}}{\sigma(x)}
\cdot \frac{1}{2\pi} \int_{-\psi(x)}^{\psi(x)} F \left ( c + \frac{\mathi v}{
\sigma(x)} \right )  \frac{e^{-v^2/2} \mathd v}{c + \mathi v / \sigma(x)}
$$
Let $G(z) = F(z)/z$. Then 
$G(c + \mathi v / \sigma(x)) = G(c) + \mathi v G'(c) /
\sigma(x)
O(v^2 / \sigma(x)^2)$. 
Plugging this into the integral the middle term vanishes
while the last term contributes a negligible amount. Thus, we end up with
$$
F(c) \cdot \frac{e^{-\Delta^2/2}}{\sqrt{2\pi}\Delta} \cdot (1 + o(1)) 
= F(c) \int_{\Delta}^{\infty} e^{-u^2/2} \frac{\mathd u}{\sqrt{2\pi}} ( 1 + o(1))
$$
as a main term for (\ref{saddle}).
To conclude it suffices to notice that $F(c) \sim F(k)$ as $T \rightarrow 
\infty$, since $c \sim k$ and $F$ is analytic. Furthermore as in the proof of
Proposition \ref{Proposition2} the restriction to $t \in A$ in (\ref{saddle})
can be replaced by $t \in [T;2T]$ because $A$ is very close in measure to $T$
(that is, $\text{meas}(A^c) \ll T (\log T)^{-k^2 - 1}$). 
 
\end{proof}

\bibliography{largedev4}
\bibliographystyle{plain}

\end{document}